\documentclass[fontsize=10pt]{scrreprt}

\usepackage[square,numbers,sort,nonamebreak]{natbib}
\usepackage{setspace}
\usepackage[english]{babel}
\usepackage{chngcntr}
\counterwithout{equation}{chapter}
\counterwithout{table}{chapter}
\usepackage[hang]{footmisc}
 at 11truept
\font\ssc=pplrc9d at 11 truept
\newcommand\qedbox{$\rlap{$\sqcap$}\sqcup$}
\usepackage[hmarginratio=1:1,bottom=2cm,top=2.5cm]{geometry}
\usepackage[automark]{scrlayer-scrpage}
\cohead{\textnormal{A counterexample to Byott's conjecture}}
\lehead{\pagemark}
\rohead{\pagemark}
\let\ceheadL\cehead
\renewcommand\cehead[1]{\ceheadL{\textnormal{#1}}}
\cehead{Massimiliano Di Matteo --- Maria Ferrara --- Marco Trombetti}
\lefoot{}
\rofoot{}

\usepackage{graphicx}
\usepackage[usenames,dvipsnames,svgnames,table]{xcolor}
\definecolor{Maroon}{cmyk}{0,0.87,0.68,0.32}
\definecolor{RoyalBlue2}{cmyk}{80,100,0,0.1}
\newcommand\auths[1]{\large\textsc{\textcolor{Maroon}{#1}}\setstretch{1.2}}
\newcommand\titl[1]{\center\linespread{1.1}\color{RoyalBlue2}\Large\textbf{#1}\color{black}\bigskip}
\renewcommand\abstract[1]{
\begin{center}{\textbf{Abstract}}\end{center}
{\linespread{1.1}\fontsize{9pt}{10pt}\selectfont #1}}

\usepackage[utf8]{inputenc}
\usepackage[sc,osf]{mathpazo}
\usepackage[euler-digits]{eulervm}
\usepackage[T1]{fontenc}
\usepackage{mathtools}
\DeclareSymbolFont{operators}{\encodingdefault}{ppl}{m}{n}
\DeclareMathAlphabet{\mathbf}{\encodingdefault}{ppl}{bx}{n}
\DeclareMathAlphabet{\mathit}{\encodingdefault}{ppl}{m}{it}
\usepackage{amsfonts,amsmath,amssymb,ragged2e,titlesec,booktabs,array,tabularx}
\usepackage[hidelinks,pdftitle={A Counterexample to Byott's Conjecture for Finite Skew Braces},pdfauthor={Massimiliano Di Matteo, Maria Ferrara, Marco Trombetti}]{hyperref}

\renewcommand{\thesection}{\arabic{section}}

\titleformat{\section}{\medskip\bigskip\normalfont\Large\bf}{\thesection}{0.5em}{}
\titleformat{\subsection}{\smallskip\bigskip\normalfont\large\bf}{\thesubsection}{0.5em}{}

\usepackage{amsthm}
\newtheoremstyle{dotless}{}{}{\itshape}{}{\bfseries}{}{1em}{}
\theoremstyle{dotless}
\newtheorem{theo}{Theorem}
\newtheorem{prop}[theo]{Proposition}
\newtheorem{lem}[theo]{Lemma}

\renewenvironment{proof}{\par\smallskip\noindent{\sc Proof \;---\;}}{\hfill\qedbox\par\medskip}
\numberwithin{theo}{section}

\DeclareOldFontCommand{\rm}{\normalfont\rmfamily}{\mathrm}
\DeclareOldFontCommand{\sf}{\normalfont\sffamily}{\mathsf}
\DeclareOldFontCommand{\tt}{\normalfont\ttfamily}{\mathtt}
\DeclareOldFontCommand{\bf}{\normalfont\bfseries}{\mathbf}
\DeclareOldFontCommand{\it}{\normalfont\itshape}{\mathit}
\DeclareOldFontCommand{\sl}{\normalfont\slshape}{\@nomath\sl}
\DeclareOldFontCommand{\sc}{\normalfont\scshape}{\@nomath\sc}

\DeclareMathOperator{\GL}{GL}
\DeclareMathOperator{\PSL}{PSL}
\DeclareMathOperator{\Inn}{Inn}
\DeclareMathOperator{\Aut}{Aut}
\DeclareMathOperator{\Hol}{Hol}

\newcommand{\F}{\mathbf F}
\newcommand{\op}{\mathrm{op}}
\newcommand{\st}{\mathbin{\star}}

\begin{document}
\setheadsepline{1pt}[\color{black}]

\titl{A Counterexample to Byott's Conjecture for Finite Skew Braces}

\auths{Massimiliano Di Matteo --- Maria Ferrara --- Marco Trombetti}

\thispagestyle{empty}
\justify
\setstretch{1.05}
\abstract{We construct a finite skew brace with soluble additive group and insoluble multiplicative group.  The multiplicative group has a quotient isomorphic to $\PSL_2(7)$, and hence the example disproves Byott's conjecture.}

\medskip
\noindent{\small\textit{Mathematics Subject Classification (2020).} 16T25, 20D10, 12F10.\par}
\smallskip
\noindent{\small\textit{Keywords.} Skew brace, regular subgroup, holomorph, Hopf--Galois structure, bijective cocycle, soluble group.\par}
\bigskip

\setstretch{1.1}
\fontsize{11pt}{12.5pt}\selectfont
\section{Introduction}

Let $N$ be a finite soluble group.  Byott's conjecture asserts that every
regular subgroup of $\Hol(N)=N\rtimes\Aut(N)$ is soluble.  The question arose
from Byott's work on solubility criteria for Hopf--Galois structures
\cite{Byott2015} and was subsequently isolated and substantially reduced in
purely group-theoretic form in \cite{Byott2024}.  In the language of skew
braces, the problem was recorded in \cite[Question~1.25]{SmoktunowiczVendramin2018}
and \cite[Problem~2.46]{Vendramin2019}; see also the group-theoretic survey
\cite{BallesterEtAl2023}.  It continued to be stated as open in later work,
for example \cite{DameleErcan2026,DameleTwoSided2026,
GorshkovNasybullov2021,LiNasybullovZadvornov2025,
StefanelloTrappeniers2023b}.

The conjecture has equivalent formulations in several areas.  A subgroup of
$\Hol(N)$ is regular when its natural action on $N$ is free and transitive.
If $G$ and $N$ have the same finite order, a regular embedding of $G$ into
$\Hol(N)$ is equivalently given by a homomorphism
$\lambda:G\to\Aut(N)$ and a bijection $b:G\to N$ satisfying
$$
 b(gh)=b(g)\lambda_g(b(h)).
$$
Thus the problem asks whether an insoluble group can occur as the source of a
bijective non-abelian $1$-cocycle with soluble target.

This formulation is directly connected with Hopf--Galois theory.  By the
Greither--Pareigis correspondence, Hopf--Galois structures on a finite
Galois extension are described by regular permutation groups normalized by
the left regular representation of the Galois group
\cite{GreitherPareigis1987}.  In the holomorph formulation, a Galois group
$G$ admits a Hopf--Galois structure of type $N$ precisely when $\Hol(N)$
contains a regular subgroup isomorphic to $G$.  Consequently, Byott's
conjecture predicts that an insoluble finite Galois group cannot admit a
Hopf--Galois structure of soluble type.

The same data are equivalent to a skew brace.  A \emph{skew brace} is a set $B$ with
two group operations, denoted by $+$ and $\circ$, such that
\begin{equation}\label{eq:brace-identity}
 a\circ(b+c)=(a\circ b)-a+(a\circ c)
\end{equation}
for all $a,b,c\in B$.  The group $(B,+)$ need not be abelian; throughout,
$-a$ denotes the inverse of $a$ in $(B,+)$.  If $x,y\in B$, we write
$\Inn_x^B(y)=x+y-x$ for conjugation in the additive group.  These conventions
fix the order of every term in formulas involving $+$.
Skew braces provide the algebraic framework for non-degenerate
set-theoretic solutions of the Yang--Baxter equation
\cite{GuarnieriVendramin2017}.  Their connections with regular subgroups,
bijective cocycles, factorizations and Hopf--Galois structures are developed
in \cite{SmoktunowiczVendramin2018,StefanelloTrappeniers2023}; a systematic
account is given in the monograph of Ced\'o and Vendramin
\cite{CedoVendramin2026}.

A substantial collection of affirmative special cases made the conjecture
plausible.  It holds when the additive group is nilpotent
\cite{CedoSmoktunowiczVendramin2019}; for cube-free orders, for several
additional families of orders and for all orders at most $2000$
\cite{TsangQin2020}; and when the order is not divisible by $3$
\cite{GorshkovNasybullov2021}.  It holds for finite two-sided skew braces
\cite{Nasybullov2019}; subsequent structural results for two-sided skew braces
appear in \cite{Trappeniers2023,DameleErcan2026}, and every finite quotient of
the multiplicative group of an arbitrary two-sided skew brace of soluble type
is soluble \cite{DameleTwoSided2026}.  The conjecture also holds for bi-skew braces
\cite{StefanelloTrappeniers2023b}, when the derived subgroup of the additive
group is cyclic \cite{LiNasybullovZadvornov2025}, when the additive group is
a $Z$-group \cite{DameleZGroup2026}, when it is a soluble CCS group
\cite{DameleMastrogiacomo2026}, and when its Sylow $2$-subgroup is cyclic
\cite{DameleCyclicSylow2026}.  A connected locally compact analogue is
proved in \cite{DameleLoi2026}.

The example below follows the structural restrictions established in
\cite{Byott2024}.  Its insoluble quotient is
$\GL_3(2)\simeq\PSL_2(7)$, its lowest quotient is the natural module
$\F_2^3$, and its odd layers arise from the Frobenius subgroup
$C_7\rtimes C_3$.  The construction uses an affine copy of $\GL_3(2)$, two
induced cocycles in characteristics $3$ and $7$, and the opposite of an
iterated semidirect product.  Passing to the opposite group is essential for
the orientation of the cocycle.

The example was constructed explicitly in GAP, using the package
\texttt{YangBaxter} \cite{GAP4,YangBaxter}.  The defining maps, the finite
factorizations in $\GL_3(2)$ and all the properties stated below were also
checked computationally.  The proof is nevertheless given directly in
coordinates; the computational checks are mentioned only next to the
calculations to which they apply.

We write $\F_q$ for the field with $q$ elements; in particular, $\F_2$
is the field with two elements.  If $X$ is a group, then $X^{\op}$ denotes
the opposite group: it has the same underlying set as $X$, and its product is
$x\cdot_{\op}y=yx$.

\bigskip
\noindent{\bf Main Theorem}\quad
{\it There exists a finite skew brace $(B,+,\circ)$ such that $(B,+)$ is
soluble and $(B,\circ)$ is insoluble.  More precisely,
$$
 (B,+)\simeq
 \left(\F_7^{\,8}\rtimes
       (\F_3^{\,7}\rtimes\F_2^{\,3})\right)^{\op},
$$
whereas $(B,\circ)$ has a normal soluble subgroup $H$ satisfying
$$
 (B,\circ)/H\simeq\GL_3(2)\simeq\PSL_2(7).
$$
In particular, Byott's conjecture is false.}
\bigskip

\section{Construction of the example}

Put $E=\F_2^3$, written as column vectors, and enumerate its elements by
$$
\begin{array}{llll}
 e_0=(0,0,0)^{\mathsf T},&e_1=(0,0,1)^{\mathsf T},
 &e_2=(0,1,0)^{\mathsf T},&e_3=(0,1,1)^{\mathsf T},\\
 e_4=(1,0,0)^{\mathsf T},&e_5=(1,0,1)^{\mathsf T},
 &e_6=(1,1,0)^{\mathsf T},&e_7=(1,1,1)^{\mathsf T}.
\end{array}
$$
Put $L=\GL_3(2)$.  Let
$$
 A=\begin{pmatrix}0&0&1\\0&1&0\\1&0&0\end{pmatrix},
 \qquad
 B=\begin{pmatrix}0&0&1\\0&1&1\\1&1&0\end{pmatrix}.
$$
In $\operatorname{AGL}(E)$ we represent an affine transformation by a pair
$(M,z)\in L\times E$, acting by $x\mapsto Mx+z$, with multiplication
\begin{equation}\label{eq:affine-law}
 (M,z)(N,t)=(MN,Mt+z).
\end{equation}
Set
$$
 \widehat A=(A,e_2),\qquad \widehat B=(B,e_4),
 \qquad \widehat Y=\widehat A\widehat B^{\,2},
 \qquad \widehat L=\langle\widehat A,\widehat B\rangle.
$$

\begin{lem}\label{lem:affine-copy}
The projection $\widehat L\to L$ is an isomorphism.
\end{lem}

\begin{proof}
We use the standard presentation
\begin{equation}\label{eq:psl-presentation}
 \PSL_2(7)=
 \langle x,y\mid x^2=y^3=(xy)^7=[x,y]^4=1\rangle,
\end{equation}
where $[x,y]=x^{-1}y^{-1}xy$.
Direct multiplication with \eqref{eq:affine-law} gives
$$
 \widehat Y=
 \left(
 \begin{pmatrix}0&1&0\\1&0&1\\1&1&0\end{pmatrix},e_7
 \right),
 \qquad
 \widehat A\widehat Y=
 \left(
 \begin{pmatrix}1&1&0\\1&0&1\\0&1&0\end{pmatrix},e_5
 \right).
$$
The four relators in \eqref{eq:psl-presentation} are checked as follows.
First $Ae_2=e_2$, hence
$$
 \widehat A^{\,2}=(A^2,Ae_2+e_2)=(I,e_0).
$$
Using the displayed expression for $\widehat Y$, one obtains
$$
 \widehat Y^{\,2}=
 \left(
 \begin{pmatrix}1&0&1\\1&0&0\\1&1&1\end{pmatrix},e_3
 \right),
 \qquad
 \widehat Y^{\,3}=(I,e_0).
$$
Moreover $\widehat A\widehat Y=\widehat B^{\,2}$.  The orbit of $e_0$ under $\widehat B$ is
$$
 e_0\mapsto e_4\mapsto e_5\mapsto e_3\mapsto e_1
 \mapsto e_2\mapsto e_7\mapsto e_0,
$$
and $e_6$ is fixed.  Hence $\widehat B^{\,7}=1$, so
$(\widehat A\widehat Y)^7=1$.
Finally
$$
 [\widehat A,\widehat Y]=
 \left(
 \begin{pmatrix}0&0&1\\1&1&0\\1&0&0\end{pmatrix},e_6
 \right),
$$
and its successive powers have translation parts
$$
 e_6,\quad e_7,\quad e_3,\quad e_0,
$$
with fourth linear power equal to $I$.  Thus
$[\widehat A,\widehat Y]^4=1$.

The verified relators define a homomorphism
$$
 \Psi:\PSL_2(7)\longrightarrow\operatorname{AGL}(E),
 \qquad x\longmapsto\widehat A,
 \qquad y\longmapsto\widehat Y.
$$
Its image contains $\widehat A$.  It also contains $\widehat B$, because
$\widehat A\widehat Y=\widehat B^{\,2}$ and, as $\widehat B$ has order $7$,
$$
 \widehat B=(\widehat B^{\,2})^4=(\widehat A\widehat Y)^4.
$$
Therefore $\operatorname{im}\Psi=\widehat L$.  The homomorphism $\Psi$ is
nontrivial, so it is injective because $\PSL_2(7)$ is simple.  Hence
$$
 \widehat L\simeq\PSL_2(7)
 \qquad\text{and}\qquad
 |\widehat L|=168.
$$

Now let $\varphi:\widehat L\to L$ be the restriction of the projection
$(M,z)\mapsto M$.  Its kernel is normal in the simple group $\widehat L$.
Since $\varphi(\widehat A)=A\ne I$, the kernel is not all of $\widehat L$;
hence $\ker\varphi=1$.  Thus $|\varphi(\widehat L)|=168$.  On the other
hand,
$$
 |L|=|\GL_3(2)|=(2^3-1)(2^3-2)(2^3-2^2)=7\cdot6\cdot4=168.
$$
Consequently $\varphi(\widehat L)=L$, and $\varphi:\widehat L\to L$ is
an isomorphism.
\end{proof}

By Lemma~\ref{lem:affine-copy}, the projection
$\varphi:\widehat L\to L$ is an isomorphism.  Thus each $g\in L$ has a
unique lift
$$
 \widehat g=\varphi^{-1}(g)\in\widehat L.
$$
Every element of $\operatorname{AGL}(E)$ has a unique expression as a pair
$(M,z)$, so there is a unique vector $\pi(g)\in E$ such that
$$
 \widehat g=(g,\pi(g)).
$$
This defines a function $\pi:L\to E$; equivalently, $\pi(g)$ is the image
of $e_0$ under the affine transformation $\widehat g$.  From
\eqref{eq:affine-law} we obtain
$$
 \widehat g\,\widehat h
 =(g,\pi(g))(h,\pi(h))
 =(gh,g\pi(h)+\pi(g)).
$$
Since $\widehat g\,\widehat h=\widehat{gh}$, comparison of translation
parts gives
\begin{equation}\label{eq:pi-cocycle}
 \pi(gh)=\pi(g)+g\pi(h).
\end{equation}
We transport the natural action of $\widehat L$ on $E$ through the inverse
of $\varphi$.  Explicitly, for $g\in L$ and $x\in E$, define
\begin{equation}\label{eq:affine-action}
 g\st x=\widehat g(x)=gx+\pi(g).
\end{equation}
Then
\begin{align*}
 (gh)\st x
 &=ghx+\pi(gh)\\
 &=ghx+\pi(g)+g\pi(h)\\
 &=g\bigl(hx+\pi(h)\bigr)+\pi(g)\\
 &=g\st(h\st x),
\end{align*}
so \eqref{eq:affine-action} is a group action.
The seven-cycle displayed in the proof omits only the fixed point
$e_6$ of $\widehat B$.  Since
$\widehat A(e_1)=e_6$, the group $\widehat L$ joins
that point to the seven-cycle.  Hence the affine action is transitive on all
eight points of $E$.

Let $F$ be the stabilizer of $e_0$.  Then $|F|=168/8=21$.  Explicitly, direct multiplication of the affine generators gives
$$
 \widehat A\widehat B\widehat A\widehat B^{4}
 \widehat A\widehat B\widehat A\widehat B=(U,e_0),
 \qquad
 \widehat B^{2}\widehat A\widehat B^{4}
 \widehat A\widehat B=(Y,e_0),
$$
where
$$
 U=\begin{pmatrix}0&0&1\\1&1&0\\0&1&1\end{pmatrix},
 \qquad
 Y=\begin{pmatrix}0&0&1\\1&0&1\\1&1&0\end{pmatrix}.
$$
Thus both affine elements fix $e_0$, so their linear parts belong to $F$.
They satisfy
\begin{equation}\label{eq:F21}
 U^7=Y^3=1,
 \qquad YUY^{-1}=U^2.
\end{equation}
For example, direct multiplication gives
$$
 U^2=\begin{pmatrix}0&1&1\\1&1&1\\1&0&1\end{pmatrix},
 \qquad
 Y^2=\begin{pmatrix}1&1&0\\1&1&1\\1&0&0\end{pmatrix},
$$
and
$$
 YUY^2=
 \begin{pmatrix}0&1&1\\1&1&1\\1&0&1\end{pmatrix}=U^2.
$$
Continuing the same multiplication gives $U^7=I$ and $Y^3=I$.
The subgroup $\langle U,Y\rangle$ contains the $21$ distinct elements
$U^mY^k$, because $\langle U\rangle$ is normal of order $7$ and
$\langle U\rangle\cap\langle Y\rangle=1$.  Since it is contained in
the stabilizer of order $21$, we obtain
$$
 F=\langle U,Y\rangle\simeq C_7\rtimes C_3.
$$
Every element of $F$ therefore has a unique expression $U^mY^k$ with
$m\in\F_7$ and $k\in\F_3$.  Here and below, powers indexed by elements of
$\F_7$ or $\F_3$ are well defined modulo $7$ or $3$, respectively; thus
$U^m$ and $Y^k$ do not depend on the chosen integer representatives.
Since $g\st e_0=\pi(g)$ and the action of $L$ on $E$ is
transitive, the map $\pi:L\to E$ is surjective.  Moreover,
$F=\{g\in L:\pi(g)=e_0\}$, and therefore the fibres of $\pi$ are
the right cosets of $F$.  Choose $t_x\in L$ with $\pi(t_x)=x$ for each
$x\in E$, taking $t_{e_0}=1$.  The choices used below are listed in
Table~\ref{tab:transversal}; every word is written directly in the matrices
$A$ and $B$.

\begin{table}[ht]
\centering
\small
\renewcommand{\arraystretch}{1.12}
\begin{tabular}{c@{\qquad}c|c@{\qquad}c}
\hline
$x$ & $t_x$ & $x$ & $t_x$\\
\hline
$e_0$ & $1$ & $e_4$ & $B$\\
$e_1$ & $BAB$ & $e_5$ & $B^3AB^4A$\\
$e_2$ & $A$ & $e_6$ & $BABAB$\\
$e_3$ & $AB^2AB^4A$ & $e_7$ & $AB^2AB^5$\\
\hline
\end{tabular}
\caption{A transversal for the eight right cosets of $F$ in $L$.}
\label{tab:transversal}
\end{table}

For example, $\pi(A)=e_2$ and $\pi(B)=e_4$.  Using
\eqref{eq:pi-cocycle},
$$
 \pi(BAB)
 =\pi(B)+B\pi(A)+BA\pi(B)
 =e_4+Be_2+BAe_4=e_1,
$$
so the entry $t_{e_1}=BAB$ has the required property.  Similarly, direct multiplication in $\operatorname{AGL}(E)$ gives
$$
 \widehat B\widehat A\widehat B\widehat A\widehat B
 =(BABAB,e_6),
$$
so $\pi(BABAB)=e_6$ and the entry $t_{e_6}=BABAB$ is correct.  The remaining entries are checked
in exactly the same way.

For $g\in L$ and $x\in E$, both $gt_x$ and $t_{g\st x}$ send
$e_0$ to $g\st x$: indeed,
$$
 (gt_x)\st e_0=g\st(t_x\st e_0)=g\st x,
 \qquad
 t_{g\st x}\st e_0=g\st x.
$$
Hence $t_{g\st x}^{-1}gt_x$ fixes $e_0$ and belongs to its stabilizer
$F$.  We may therefore define
\begin{equation}\label{eq:sigma}
 \sigma(g,x)=t_{g\st x}^{-1}gt_x\in F.
\end{equation}
Write uniquely
\begin{equation}\label{eq:mk}
 \sigma(g,x)=U^{m(g,x)}Y^{k(g,x)},
 \qquad m(g,x)\in\F_7,
 \quad k(g,x)\in\F_3.
\end{equation}

\begin{lem}\label{lem:schreier}
For all $g,h\in L$ and $x\in E$,
\begin{equation}\label{eq:schreier-identity}
 \sigma(gh,x)=\sigma(g,h\st x)\sigma(h,x).
\end{equation}
Consequently
\begin{align}
 k(gh,x)&=k(g,h\st x)+k(h,x),\label{eq:k-relation}\\
 m(gh,x)&=m(g,h\st x)+2^{k(g,h\st x)}m(h,x).
 \label{eq:m-relation}
\end{align}
\end{lem}

\begin{proof}
Insert $t_{h\st x}t_{h\st x}^{-1}$ between $g$ and $h$:
$$
 t_{gh\st x}^{-1}ght_x
 =\bigl(t_{g\st(h\st x)}^{-1}gt_{h\st x}\bigr)
  \bigl(t_{h\st x}^{-1}ht_x\bigr).
$$
This proves equation~\eqref{eq:schreier-identity}.  Since
$Y^kU^nY^{-k}=U^{2^kn}$, one has
$$
 U^mY^kU^nY^\ell
 =U^{m+2^kn}Y^{k+\ell},
$$
and comparison of the two exponents gives equations~\eqref{eq:k-relation}
and~\eqref{eq:m-relation}.
\end{proof}

The exponents for the generators $A$ and $B$ are listed in
Table~\ref{tab:generator-values}.  For instance,
$$
 A\st e_1=e_6,
 \qquad t_{e_6}^{-1}At_{e_1}=U^2,
$$
so $m(A,e_1)=2$ and $k(A,e_1)=0$; similarly,
$$
 B\st e_2=e_7,
 \qquad t_{e_7}^{-1}Bt_{e_2}=U^4Y,
$$
so $m(B,e_2)=4$ and $k(B,e_2)=1$.  Every other entry is obtained by the
same matrix multiplication and unique reduction to $U^mY^k$.

\begin{table}[ht]
\centering
\scriptsize
\renewcommand{\arraystretch}{1.12}
\begin{tabular}{c|ccc|ccc}
\hline
$x$ & $A\st x$ & $m(A,x)$ & $k(A,x)$
    & $B\st x$ & $m(B,x)$ & $k(B,x)$\\
\hline
$e_0$ & $e_2$ & $0$ & $0$ & $e_4$ & $0$ & $0$\\
$e_1$ & $e_6$ & $2$ & $0$ & $e_2$ & $2$ & $0$\\
$e_2$ & $e_0$ & $0$ & $0$ & $e_7$ & $4$ & $1$\\
$e_3$ & $e_4$ & $0$ & $2$ & $e_1$ & $0$ & $2$\\
$e_4$ & $e_3$ & $0$ & $1$ & $e_5$ & $0$ & $1$\\
$e_5$ & $e_7$ & $2$ & $1$ & $e_3$ & $2$ & $1$\\
$e_6$ & $e_1$ & $5$ & $0$ & $e_6$ & $5$ & $0$\\
$e_7$ & $e_5$ & $6$ & $2$ & $e_0$ & $1$ & $1$\\
\hline
\end{tabular}
\caption{The exponents associated with the two generators of $L$.}
\label{tab:generator-values}
\end{table}

The transversal in Table~\ref{tab:transversal}, the two sample reductions
above and all entries of Table~\ref{tab:generator-values} can be computed
directly in GAP from the displayed matrices.

For a field $K$, write $K^E$ for the vector space of all functions
$f:E\to K$, with pointwise addition and scalar multiplication.  Define
$$
 W=\left\{w\in\F_3^E:\sum_{x\in E}w(x)=0\right\},
 \qquad
 P=\F_7^E.
$$
Thus $P$ is the set of all functions $p:E\to\F_7$, endowed with its natural
$\F_7$-vector-space structure.  For each $g\in L$, the map
$x\mapsto g\st x$ is a permutation of $E$.  On the $\F_3$-space $W$, define a linear map
$\rho_W(g):W\to W$ by
\begin{equation}\label{eq:rhoW}
 (\rho_W(g)w)(g\st x)=w(x)
 \qquad(x\in E).
\end{equation}
Equivalently, $(\rho_W(g)w)(y)=w(g^{-1}\st y)$.  Since $g\st-$ is a
permutation of $E$, the sum of the coordinates is preserved, so
$\rho_W(g)w$ again belongs to $W$.

We next define the action on $P$.  For $t\in\F_3$, write $2^t$ for the
value at $t$ of the homomorphism $(\F_3,+)\to\F_7^\times$ sending $1$ to
$2$; this is well defined because $2$ has order $3$ in $\F_7^\times$.
For $g\in L$, define a linear map $\rho_P(g):P\to P$ by
\begin{equation}\label{eq:rhoP}
 (\rho_P(g)p)(g\st x)=2^{k(g,x)}p(x)
 \qquad(x\in E).
\end{equation}
The subscripts indicate that the two maps $\rho_W(g)$ and $\rho_P(g)$ act
on $W$ and $P$, respectively.  We now verify that
$g\mapsto\rho_W(g)$ and $g\mapsto\rho_P(g)$ are homomorphisms into
$\GL(W)$ and $\GL(P)$, respectively.  At the point $gh\st x$,
\begin{align*}
 (\rho_W(g)\rho_W(h)w)(gh\st x)
 &= (\rho_W(h)w)(h\st x)=w(x)\\
 &= (\rho_W(gh)w)(gh\st x),
\end{align*}
and, using Equation~\eqref{eq:k-relation},
\begin{align*}
 (\rho_P(g)\rho_P(h)p)(gh\st x)
 &=2^{k(g,h\st x)}(\rho_P(h)p)(h\st x)\\
 &=2^{k(g,h\st x)+k(h,x)}p(x)\\
 &=2^{k(gh,x)}p(x)
 =(\rho_P(gh)p)(gh\st x).
\end{align*}
The multiplier $2^{k(g,x)}$ in Equation~\eqref{eq:rhoP} is used precisely
in this calculation: the additive relation for $k$ in
Equation~\eqref{eq:k-relation} makes the multipliers compose.  The same
multiplier also matches the coefficient in Equation~\eqref{eq:m-relation},
which was obtained from $Y^kU^nY^{-k}=U^{2^kn}$, and this is what yields the
cocycle identity for $\beta_P$ below.

Define $\beta_W:L\to\F_3^E$ and $\beta_P:L\to P$ by
\begin{equation}\label{eq:betas}
 \beta_W(g)(g\st x)=k(g,x),
 \qquad
 \beta_P(g)(g\st x)=m(g,x).
\end{equation}
We identify a function $f:E\to\F_q$ with its coordinate vector
$$
 \bigl(f(e_0),f(e_1),\ldots,f(e_7)\bigr)
$$
with respect to the ordered list $e_0,e_1,\ldots,e_7$.  In this notation,
the values of the three maps on the generators $A$ and $B$ are
\begin{align*}
 \pi(A)&=e_2,
 &\beta_W(A)&=(0,0,0,1,2,2,0,1),
 &\beta_P(A)&=(0,5,0,0,0,6,2,2),\\
 \pi(B)&=e_4,
 &\beta_W(B)&=(1,2,0,1,0,1,0,1),
 &\beta_P(B)&=(1,0,2,2,0,0,5,4).
\end{align*}
The two vectors $\beta_W(A)$ and $\beta_W(B)$ have coordinate sum $0$, and therefore lie in $W$.

\begin{prop}\label{prop:shapiro}
For all $g,h\in L$,
\begin{align}
 \beta_W(gh)&=\beta_W(g)+\rho_W(g)\beta_W(h),
 \label{eq:betaW-cocycle}\\
 \beta_P(gh)&=\beta_P(g)+\rho_P(g)\beta_P(h).
 \label{eq:betaP-cocycle}
\end{align}
In particular $\beta_W(g)\in W$ for every $g\in L$.
\end{prop}

\begin{proof}
Evaluate both sides of \eqref{eq:betaW-cocycle} at
$gh\st x=g\st(h\st x)$ and use \eqref{eq:k-relation}.  The same evaluation,
using \eqref{eq:m-relation}, proves \eqref{eq:betaP-cocycle}.  Since $W$ is $L$-invariant, the set
$$
 S=\{g\in L:\beta_W(g)\in W\}
$$
is closed under multiplication by equation~\eqref{eq:betaW-cocycle}.
It contains $1,A,B$.  Since $L$ is finite and generated by $A,B$, the
submonoid generated by $A,B$ is all of $L$; hence $S=L$ and
$\beta_W(g)\in W$ for every $g\in L$.

\end{proof}

For $e\in E$ and a function $f$ on $E$, put
$$
 (\tau_e f)(x)=f(x+e).
$$
Each $\tau_e$ preserves $W$, because translation permutes the eight elements
of $E$, and $\tau_e\tau_{e'}=\tau_{e+e'}$.  Thus $E$ acts on $W$ by
translations.  Define
$$
 Q=W\rtimes E
$$
by
\begin{equation}\label{eq:Q-law}
 (w,e)(w',e')=(w+\tau_e w',e+e').
\end{equation}
Let $Q$ act on $P$ by
\begin{equation}\label{eq:Q-action}
 ((w,e)\cdot p)(x)=2^{w(x)}p(x+e).
\end{equation}
Indeed, for $q=(w,e)$ and $q'=(w',e')$,
\begin{align*}
 (q\cdot(q'\cdot p))(x)
 &=2^{w(x)}2^{w'(x+e)}p(x+e+e')\\
 &=2^{(w+\tau_e w')(x)}p(x+e+e')\\
 &=((qq')\cdot p)(x).
\end{align*}
Thus we may form
$$
 N_0=P\rtimes Q.
$$
In coordinates its multiplication is
\begin{equation}\label{eq:N0-law}
 (p,w,e)(p',w',e')=
 \left(
 p+(w,e)\cdot p',\,
 w+\tau_e w',\,
 e+e'
 \right).
\end{equation}
The chain
$$
 1<P<P\rtimes W<N_0
$$
has abelian factors $P,W,E$, so $N_0$ is soluble of derived length at most
three.  Moreover
\begin{equation}\label{eq:N-order}
 |N_0|=|P||W||E|=7^8\,3^7\,2^3=100860958296.
\end{equation}

We now use the ordinary linear action of $L$ on $E$, rather than the
affine action $\star$.  It induces, on both $P$ and $W$, the operation
$$
 ({}^gf)(x)=f(g^{-1}x).
$$
This action is needed to define automorphisms of $N_0$.  For each $g\in L$,
define a map
\begin{equation}\label{eq:theta}
 \vartheta_g:N_0\longrightarrow N_0,
 \qquad
 \vartheta_g(p,w,e)=({}^gp,{}^gw,ge).
\end{equation}
We prove that $\vartheta_g$ is an automorphism of $N_0$.  First, the map
$(w,e)\mapsto({}^gw,ge)$ is an automorphism of $Q$.  For every $f:E\to\F_3$ one has
$$
 {}^g(\tau_e f)(x)=f(g^{-1}x+e)
 =\tau_{ge}({}^gf)(x).
$$
Consequently
$$
 ({}^gw,ge)({}^gw',ge')
 =({}^g(w+\tau_e w'),g(e+e')).
$$
It remains to check compatibility with the action of $Q$ on $P$.  At
$x\in E$,
\begin{align*}
 {}^g((w,e)\cdot p)(x)
 &=2^{w(g^{-1}x)}p(g^{-1}x+e),\\
 (({}^gw,ge)\cdot{}^gp)(x)
 &=2^{w(g^{-1}x)}p(g^{-1}(x+ge)),
\end{align*}
and these expressions coincide because
$g^{-1}(x+ge)=g^{-1}x+e$.  Thus $\vartheta_g\in\Aut(N_0)$.  The
identity $\vartheta_{gh}=\vartheta_g\vartheta_h$ follows directly from
precomposition, so $g\mapsto\vartheta_g$ is a homomorphism
$L\to\Aut(N_0)$.

Put
\begin{equation}\label{eq:c-g}
 c_g=(\beta_P(g),\beta_W(g),\pi(g))\in N_0.
\end{equation}
For the identity element $1\in L$, the unique lift is
$\widehat 1=(I,e_0)$, so $\pi(1)=e_0$.  Moreover
$t_x^{-1}1t_x=1$ for every $x\in E$, and hence
$m(1,x)=k(1,x)=0$.  Therefore
$$
 \beta_W(1)=0,\qquad \beta_P(1)=0,\qquad
 c_1=(0,0,e_0),
$$
where the first two zeros denote the zero functions in $P$ and $W$.
For the generators $A,B\in L$, the vectors displayed above give explicitly
\begin{align*}
 c_A&=\bigl((0,5,0,0,0,6,2,2),
             (0,0,0,1,2,2,0,1),e_2\bigr),\\
 c_B&=\bigl((1,0,2,2,0,0,5,4),
             (1,2,0,1,0,1,0,1),e_4\bigr).
\end{align*}
Here the first coordinate is read in $P=\F_7^E$, the second in
$W\leq\F_3^E$, and the third in $E=\F_2^3$.

\begin{prop}\label{prop:N0-cocycle}
For all $g,h\in L$,
\begin{equation}\label{eq:N0-cocycle}
 c_{gh}=c_g\,\vartheta_g(c_h)
 \qquad\text{in }N_0.
\end{equation}
\end{prop}

\begin{proof}
The $E$-coordinate is \eqref{eq:pi-cocycle}.  For the $W$-coordinate, note
that
$$
 \tau_{\pi(g)}({}^gf)(x)
 =f(g^{-1}(x+\pi(g)))
 =(\rho_W(g)f)(x).
$$
Thus the $W$-coordinate of $c_g\vartheta_g(c_h)$ is
$$
 \beta_W(g)+\rho_W(g)\beta_W(h)=\beta_W(gh).
$$
For the $P$-coordinate, \eqref{eq:Q-action} gives
\begin{align*}
 &\bigl((\beta_W(g),\pi(g))\cdot{}^g\beta_P(h)\bigr)(x)\\
 &\qquad=2^{\beta_W(g)(x)}
 \beta_P(h)(g^{-1}(x+\pi(g)))
 =(\rho_P(g)\beta_P(h))(x).
\end{align*}
Hence the $P$-coordinate is
$\beta_P(g)+\rho_P(g)\beta_P(h)=\beta_P(gh)$.

\end{proof}

Let
$$
 N=N_0^{\op}.
$$
From this point onward, $+$ denotes the generally noncommutative group law of
$N$, and $-n$ denotes the inverse of $n$ in $(N,+)$.  To avoid any ambiguity,
write $x\cdot_0y$ for the product of $x$ and $y$ in $N_0$.  The additive law
is then
\begin{equation}\label{eq:opposite-law}
 n+m=m\cdot_0 n
 \qquad(n,m\in N).
\end{equation}
Thus the order of the two factors is reversed when the sum is evaluated in
$N_0$, and $-n=n^{-1}$ when the right-hand side is computed in $N_0$.  Additions
inside coordinate triples continue to denote the vector-space operations in
$P$, $W$ and $E$.  The map $n\mapsto n^{-1}$ is an isomorphism
$N_0\to(N,+)$, so $(N,+)$ is soluble and has the order in
Equation~\eqref{eq:N-order}.  Every automorphism of $N_0$, in particular
every $\vartheta_g$, is also an automorphism of the opposite group $N$.

For $x\in N$, the notation $\Inn_x^N$ always refers to the additive group:
$$
 \Inn_x^N(y)=x+y-x.
$$
In particular, $\Inn_{-c_g}^N(y)=(-c_g)+y+c_g$.

Equation \eqref{eq:N0-cocycle} becomes
\begin{equation}\label{eq:op-cocycle}
 c_{gh}=\vartheta_g(c_h)+c_g
 \qquad\text{in }N.
\end{equation}
Define
\begin{equation}\label{eq:lambda-g}
 \lambda_g=\Inn^N_{-c_g}\vartheta_g,
 \qquad
 \lambda_g(n)=(-c_g)+\vartheta_g(n)+c_g.
\end{equation}

\begin{lem}\label{lem:lambda-hom}
The map $g\mapsto\lambda_g$ is a homomorphism
$L\to\Aut(N)$, and
\begin{equation}\label{eq:representative-cocycle}
 c_{gh}=c_g+\lambda_g(c_h).
\end{equation}
Moreover, for the representatives $c_g$,
\begin{equation}\label{eq:eta-representatives}
 \eta_g:=\Inn^N_{c_g}\lambda_g=\vartheta_g.
\end{equation}
\end{lem}

\begin{proof}
Using \eqref{eq:op-cocycle},
\begin{align*}
 c_g+\lambda_g(c_h)
 &=c_g+(-c_g)+\vartheta_g(c_h)+c_g\\
 &=\vartheta_g(c_h)+c_g=c_{gh},
\end{align*}
which proves \eqref{eq:representative-cocycle}.  Also
\begin{align*}
 \lambda_g\lambda_h
 &=\Inn^N_{-c_g}\vartheta_g
   \Inn^N_{-c_h}\vartheta_h\\
 &=\Inn^N_{(-c_g)+\vartheta_g(-c_h)}\vartheta_{gh}\\
 &=\Inn^N_{-c_{gh}}\vartheta_{gh}=\lambda_{gh}.
\end{align*}
Finally \eqref{eq:eta-representatives} follows immediately from
\eqref{eq:lambda-g}.  Notice that this equality is asserted for the representatives $c_g$; for a general element $h+c_g$ the corresponding
$\eta$ is $\Inn_h^N\vartheta_g$.

\end{proof}

Fix $e_0\in E$ and define
\begin{equation}\label{eq:H}
 H=\{(p,w,e_0)\in N_0:p(e_0)=0,\ w(e_0)=0\}.
\end{equation}
This is a subgroup of $N_0$, and hence also of $N=N_0^{\op}$.  Indeed, if
$h=(p,w,e_0)$ and $h'=(p',w',e_0)$ satisfy the two vanishing conditions, then
equation~\eqref{eq:N0-law} gives
$$
 hh'=\bigl(p+2^w p',w+w',e_0\bigr),
$$
where exponentiation and multiplication are pointwise; explicitly,
$$
 (2^wp')(x)=2^{w(x)}p'(x).
$$
Its two function coordinates vanish at $e_0$.  Moreover
$$
 (p,w,e_0)^{-1}=\bigl(-2^{-w}p,-w,e_0\bigr),
$$
which has the same property.  Finally, $H$ is $\vartheta(L)$-invariant because every linear
transformation $g\in L=\GL_3(2)$ fixes $e_0$, and therefore
$({}^gf)(e_0)=f(e_0)$ for every function $f$ on $E$.

There are $7^7$ functions $p\in P$ satisfying $p(e_0)=0$.  The
conditions $w(e_0)=0$ and $\sum_xw(x)=0$ leave six free coordinates in
$\F_3$.  Therefore
\begin{equation}\label{eq:H-order}
 |H|=7^7 3^6=600362847,
 \qquad |N:H|=168.
\end{equation}

For $n=(p,w,e)\in N_0$, put
\begin{equation}\label{eq:q-map}
 q(n)=(e,w(e),p(e))\in E\times\F_3\times\F_7.
\end{equation}

\begin{lem}\label{lem:coset-parameters}
For each $g\in L$,
$$
 c_gH=q^{-1}\bigl(q(c_g)\bigr)
$$
and
\begin{equation}\label{eq:c-parameter}
 q(c_g)=\bigl(\pi(g),k(g,e_0),m(g,e_0)\bigr).
\end{equation}
The $168$ triples $q(c_g)$, with $g\in L$, are pairwise distinct.
\end{lem}

\begin{proof}
Let $e=\pi(g)$ and let $h=(p',w',e_0)\in H$.  By \eqref{eq:N0-law}, the
$E$-coordinate of $c_gh$ is $e$, while at the coordinate $e$ one has
\begin{align*}
 (\beta_W(g)+\tau_e w')(e)&=\beta_W(g)(e)+w'(e_0)
 =\beta_W(g)(e),\\
 \bigl(\beta_P(g)+(\beta_W(g),e)\cdot p'\bigr)(e)
 &=\beta_P(g)(e)+2^{\beta_W(g)(e)}p'(e_0)\\
 &=\beta_P(g)(e).
\end{align*}
Hence $c_gH\subseteq q^{-1}(q(c_g))$.  Once the three displayed coordinates
are fixed, the remaining coordinates consist of seven free entries in
$\F_7$ and six free entries in $\F_3$.  Thus
$$
 |q^{-1}(q(c_g))|=7^7 3^6=|H|,
$$
so equality holds.

At $x=e_0$, equation \eqref{eq:sigma} gives
$$
 t_{\pi(g)}^{-1}g=U^{m(g,e_0)}Y^{k(g,e_0)}.
$$
Consequently every $g\in L$ has a unique expression
$$
 g=t_xU^mY^k,
 \qquad (x,m,k)\in E\times\F_7\times\F_3.
$$
Equation~\eqref{eq:c-parameter} associates to $g$ the reordered triple
$(x,k,m)$.  Therefore the $8\cdot3\cdot7=168$ triples are pairwise
distinct.

\end{proof}

In the opposite group $N$, the right cosets $c_gH$ of $N_0$ are the left
cosets $H+c_g$.  Lemma~\ref{lem:coset-parameters} and
\eqref{eq:H-order} give the disjoint decomposition
\begin{equation}\label{eq:coset-partition}
 N=\bigsqcup_{g\in L}H+c_g.
\end{equation}

Form the semidirect product
\begin{equation}\label{eq:G}
 G=H\rtimes_{\vartheta}L,
\end{equation}
where $H$ is regarded as a subgroup of the additive group $(N,+)$.  Thus
\begin{equation}\label{eq:G-law}
 (h,g)(k,\ell)=
 \bigl(h+\vartheta_g(k),g\ell\bigr).
\end{equation}
Its order is
\begin{equation}\label{eq:G-order}
 |G|=|H||L|=7^7 3^6\cdot168
 =7^8 3^7 2^3=|N|.
\end{equation}
Since $G/H\simeq L\simeq\PSL_2(7)$, the group $G$ is insoluble.

Define
\begin{equation}\label{eq:b}
 b:G\longrightarrow N,
 \qquad b(h,g)=h+c_g,
\end{equation}
and
\begin{equation}\label{eq:lambda-full}
 \lambda_{(h,g)}=\lambda_g.
\end{equation}
The map in \eqref{eq:lambda-full} is a homomorphism because the projection
$G\to L$ and $g\mapsto\lambda_g$ are homomorphisms.

\begin{prop}\label{prop:bijective-cocycle}
The map $b$ is bijective and satisfies
\begin{equation}\label{eq:full-cocycle}
 b(xy)=b(x)+\lambda_x(b(y))
 \qquad(x,y\in G).
\end{equation}
\end{prop}

\begin{proof}
For fixed $g$, the restriction of $b$ to $H\times\{g\}$ is a bijection onto
$H+c_g$.  The disjoint partition \eqref{eq:coset-partition} therefore
proves that $b$ is bijective.

Let $x=(h,g)$ and $y=(k,\ell)$.  Using \eqref{eq:G-law} and
\eqref{eq:op-cocycle},
\begin{align*}
 b(xy)
 &=h+\vartheta_g(k)+c_{g\ell}\\
 &=h+\vartheta_g(k)+
   \vartheta_g(c_\ell)+c_g.
\end{align*}
On the other hand,
\begin{align*}
 b(x)+\lambda_x(b(y))
 &=(h+c_g)+
   \bigl((-c_g)+
   \vartheta_g(k+c_\ell)+c_g\bigr)\\
 &=h+\vartheta_g(k)+
   \vartheta_g(c_\ell)+c_g.
\end{align*}
The two expressions are identical.
\end{proof}

Transport the group law of $G$ to $N$ through the bijection $b$.  Equivalently,
every element of $N$ has a unique expression $h+c_g$, and we define
\begin{equation}\label{eq:circ-explicit}
 (h+c_g)\circ(k+c_\ell)
 =h+\vartheta_g(k)+c_{g\ell}.
\end{equation}
Formula~\eqref{eq:circ-explicit} is the transport of the
semidirect-product law~\eqref{eq:G-law}.  Hence $(N,\circ)$ is a group and
$(N,\circ)\simeq G$.  Through the same bijection, for the unique expression
$a=h+c_g$ we write
$$
 \lambda_a:=\lambda_{(h,g)}=\lambda_g.
$$
Since $c_1=(0,0,e_0)$ and the identity of $H$ is the zero element of
$(N,+)$, formula~\eqref{eq:circ-explicit} shows that the identity of
$(N,\circ)$ is the same element $0$ as the identity of $(N,+)$.

\begin{theo}\label{thm:brace}
The operations $+$ and $\circ$ make $(N,+,\circ)$ into a skew brace.
For $a=h+c_g$ and every $z\in N$,
\begin{equation}\label{eq:brace-lambda}
 \lambda_a(z)=(-c_g)+\vartheta_g(z)+c_g,
\end{equation}
and
\begin{equation}\label{eq:a-circ-z}
 a\circ z=a+\lambda_a(z).
\end{equation}
\end{theo}

\begin{proof}
Equation \eqref{eq:a-circ-z} is exactly the cocycle identity
\eqref{eq:full-cocycle}, transported through $b$.  Since every $\lambda_a$
is an automorphism of $(N,+)$, for $a,b,c\in N$ we obtain
\begin{align*}
 a\circ(b+c)
 &=a+\lambda_a(b+c)\\
 &=a+\lambda_a(b)+\lambda_a(c)\\
 &=(a+\lambda_a(b))+(-a)+(a+\lambda_a(c))\\
 &=(a\circ b)-a+(a\circ c).
\end{align*}
This is precisely the skew-brace identity in equation~\eqref{eq:brace-identity}.

\end{proof}

Combining Equations~\eqref{eq:N-order} and \eqref{eq:G-order} with
Theorem~\ref{thm:brace}, the additive group is soluble, whereas the
multiplicative group has the insoluble quotient $\PSL_2(7)$.  More directly,
the map
$$
 G\longrightarrow\Hol(N),\qquad
 x\longmapsto\bigl(b(x),\lambda_x\bigr),
$$
is a homomorphism by Equation~\eqref{eq:full-cocycle}.  It is injective
because $b$ is bijective, and its orbit of $0\in N$ is $b(G)=N$; hence its
image is a regular subgroup of $\Hol(N)$ isomorphic to the insoluble group
$G$.  This proves the Main Theorem and gives a counterexample in the original
holomorph formulation.

\section*{Acknowledgements}

All authors are members of the non-profit association
``AGTA -- Advances in Group Theory and Applications''
(\texttt{www.advgrouptheory.com}) and are supported by GNSAGA (INdAM).

\bigskip\bigskip
\renewcommand{\bibsection}{\begin{flushright}\Large{REFERENCES}\\
\rule{8cm}{0.4pt}\\[0.8cm]\end{flushright}}

\bigskip\bigskip
\begin{flushleft}\rule{8cm}{0.4pt}\\\end{flushleft}

{\sloppy
\noindent Massimiliano Di Matteo

\noindent Department of Mathematics and Physics

\noindent University of Campania ``Luigi Vanvitelli''

\noindent Viale Lincoln 5, 81100 Caserta, Italy

\noindent e-mail: massimiliano.dimatteo@unicampania.it
}\bigskip\bigskip

{\sloppy
\noindent Maria Ferrara

\noindent Department of Engineering, Faculty of Engineering and Computer Science

\noindent Pegaso University, Naples, Italy

\noindent e-mail: maria.ferrara1@unipegaso.it
}\bigskip\bigskip

{\sloppy
\noindent Marco Trombetti

\noindent Department of Mathematics and Applications ``Renato Caccioppoli''

\noindent University of Naples Federico II

\noindent Monte S. Angelo University Complex, Via Cintia, 80126 Naples, Italy

\noindent e-mail: marco.trombetti@unina.it
}


\begin{thebibliography}{99}


\bibitem{BallesterEtAl2023}
A. Ballester-Bolinches, R. Esteban-Romero, P. Jim\'enez-Seral and
V. P\'erez-Calabuig,
\emph{Some group-theoretical approaches to skew left braces},
Int. J. Group Theory \textbf{12} (2023), no.~2, 99--109.

\bibitem{Byott2015}
N.~P. Byott,
\emph{Solubility criteria for Hopf--Galois structures},
New York J. Math. \textbf{21} (2015), 883--903.

\bibitem{Byott2024}
N.~P. Byott,
\emph{On insoluble transitive subgroups in the holomorph of a finite soluble group},
J. Algebra \textbf{638} (2024), 1--31.


\bibitem{CedoSmoktunowiczVendramin2019}
F. Ced\'o, A. Smoktunowicz and L. Vendramin,
\emph{Skew left braces of nilpotent type},
Proc. Lond. Math. Soc. (3) \textbf{118} (2019), 1367--1392.

\bibitem{CedoVendramin2026}
F. Ced\'o and L. Vendramin,
\emph{Groups, Radical Rings, and the Yang--Baxter Equation: A Combinatorial
Approach to Solutions}, Progress in Mathematics, vol.~361, Birkh\"auser,
Cham, 2026.

\bibitem{DameleCyclicSylow2026}
M. Damele,
\emph{Simple skew braces with cyclic Sylow subgroups},
arXiv:2607.17125 (2026).


\bibitem{DameleErcan2026}
M. Damele and G. Ercan,
\emph{Ideals and solvability in skew braces},
arXiv:2607.19955 (2026).

\bibitem{DameleLoi2026}
M. Damele and A. Loi,
\emph{Solvability and rigidity for topological skew braces},
arXiv:2605.07609 (2026).

\bibitem{DameleZGroup2026}
M. Damele,
\emph{Finite skew braces whose additive group is a $Z$-group},
arXiv:2603.22980 (2026).

\bibitem{DameleMastrogiacomo2026}
M. Damele and F. Mastrogiacomo,
\emph{Finite groups in which every proper characteristic subgroup is cyclic},
Mediterr. J. Math. \textbf{23} (2026), Article 73.

\bibitem{DameleTwoSided2026}
M. Damele,
\emph{On the multiplicative group of a two-sided skew brace of solvable type},
arXiv:2603.24637 (2026).

\bibitem{GorshkovNasybullov2021}
I. Gorshkov and T. Nasybullov,
\emph{Finite skew braces with solvable additive group},
J. Algebra \textbf{574} (2021), 172--183.

\bibitem{GreitherPareigis1987}
C. Greither and B. Pareigis,
\emph{Hopf Galois theory for separable field extensions},
J. Algebra \textbf{106} (1987), 239--258.

\bibitem{GuarnieriVendramin2017}
L. Guarnieri and L. Vendramin,
\emph{Skew braces and the Yang--Baxter equation},
Math. Comp. \textbf{86} (2017), 2519--2534.

\bibitem{LiNasybullovZadvornov2025}
B. Li, T. Nasybullov and V. Zadvornov,
\emph{Comparison of addition and multiplication in a skew brace},
arXiv:2511.22322 (2025).

\bibitem{Nasybullov2019}
T. Nasybullov,
\emph{Connections between properties of the additive and the multiplicative
groups of a two-sided skew brace},
J. Algebra \textbf{540} (2019), 156--167.

\bibitem{SmoktunowiczVendramin2018}
A. Smoktunowicz and L. Vendramin,
\emph{On skew braces (with an appendix by N. Byott and L. Vendramin)},
J. Comb. Algebra \textbf{2} (2018), 47--86.

\bibitem{StefanelloTrappeniers2023}
L. Stefanello and S. Trappeniers,
\emph{On the connection between Hopf--Galois structures and skew braces},
Bull. Lond. Math. Soc. \textbf{55} (2023), 1726--1748.

\bibitem{StefanelloTrappeniers2023b}
L. Stefanello and S. Trappeniers,
\emph{On bi-skew braces and brace blocks},
J. Pure Appl. Algebra \textbf{227} (2023), Paper No.~107295.


\bibitem{Trappeniers2023}
S. Trappeniers,
\emph{On two-sided skew braces},
J. Algebra \textbf{631} (2023), 267--286.

\bibitem{TsangQin2020}
C. Tsang and C. Qin,
\emph{On the solvability of regular subgroups in the holomorph of a finite
solvable group},
Int. J. Algebra Comput. \textbf{30} (2020), 253--265.

\bibitem{Vendramin2019}
L. Vendramin,
\emph{Problems on skew left braces},
Adv. Group Theory Appl. \textbf{7} (2019), 15--37.


\bibitem{GAP4}
{\ssc The GAP Group}: ``GAP -- Groups, Algorithms, and Programming'',
Version 4.13.1 (2024), \url{https://www.gap-system.org}.

\bibitem{YangBaxter}
{\ssc L. Vendramin -- O. Konovalov}: ``YangBaxter, Combinatorial Solutions
for the Yang--Baxter Equation'', Version 0.10.7 (2025), GAP package,
\url{https://gap-packages.github.io/YangBaxter}.

\end{thebibliography}
\end{document}